\documentclass[12pt]{amsart}

\numberwithin{equation}{section}

\usepackage{amsmath,amsthm,amsfonts,eucal,}
\usepackage{xypic}
\usepackage{graphicx}
\usepackage{amssymb}
\usepackage{tikz}

\def\cb{{\mathcal B}}

\def\cf{{\mathcal F}}

\def\ch{{\mathcal H}}

\def\gb{{\mathfrak B}}

\def\gam{{\mathfrak M}}

\def\bc{{\mathbb C}}

\def\bn{{\mathbb N}}

\def\bz{{\mathbb Z}}

  \def\G{\Gamma}

\def\l{\lambda} \def\L{\Lambda}

\def\r{\rho}

\def\om{\omega} \def\Om{\Omega}

\newtheorem{thm}{Theorem}[section]

\newtheorem{lem}[thm]{Lemma}
\newtheorem{cor}[thm]{Corollary}
\newtheorem{prop}[thm]{Proposition}

\theoremstyle{definition}
\newtheorem{rem}[thm]{Remark}
\newtheorem{defin}[thm]{Definition}

\begin{document}
\title[On the monotone $C^*$-algebra]
{On the monotone $C^*$-algebra}
\author{Vitonofrio Crismale}
\address{Vitonofrio Crismale\\
Dipartimento di Matematica\\
Universit\`{a} degli studi di Bari\\
Via E. Orabona, 4, 70125 Bari, Italy}
\email{\texttt{vitonofrio.crismale@uniba.it}}
\author{Simone Del Vecchio}
\address{Simone Del Vecchio\\
Dipartimento di Matematica\\
Universit\`{a} degli studi di Bari\\
Via E. Orabona, 4, 70125 Bari, Italy} \email{{\tt
simone.delvecchio@uniba.it}}
\author{Stefano Rossi}
\address{Stefano Rossi\\
Dipartimento di Matematica\\
Universit\`{a} degli studi di Bari\\
Via E. Orabona, 4, 70125 Bari, Italy}
\email{\texttt{stefano.rossi@uniba.it}}
\date{\today}
\begin{abstract}
The concrete monotone $C^*$-algebra, that is the  (unital) $C^*$-algebra generated by
monotone independent algebraic random variables of Bernoulli type,
 is characterized abstractly
in terms of generators and relations and is shown to be UHF. Moreover, its Bratteli diagram is explicitly given, which
allows for the computation of its $K$-theory.

\end{abstract}

\maketitle
\section{Introduction}
\noindent

Unlike Classical Probability, Quantum Probability allows for several different ways in which
random variables can be independent. Indeed, an algebraic probability space is
 a pair $(\mathcal{A}, \varphi)$, where $\mathcal{A}$ is a (unital) complex $*$-algebra and $\varphi:\mathcal{A}\rightarrow\bc$ is a (normalized) positive linear functional. Now as soon as $\mathcal{A}$ is non-commutative, there will be more than one way in which such an expression as $\varphi(a_{i_1}a_{i_2}\cdots a_{i_n})$, for $a_{i_j}\in\mathcal{A}$, can be factorized. More precisely, sticking to the categorical formalisms of natural products, exactly five forms of independence can 
arise: tensor, free, Boolean, monotone, anti-monotone, \cite{S, M2}.
In particular, monotone independence for algebraic random variables 
made its first appearance in the concrete model of the so-called monotone Fock
space, \cite{L, M0, M1}, where the arcsine law was seen to be the central limit distribution for sums of position operators. Furthermore, since position operators in the discrete monotone Fock space have a
symmetric Bernoulli distribution in the vacuum state, see also \cite{CL}, the central limit alluded to above is in fact a
non-commutative instance of the de Moivre-Laplace theorem.\\
In this note we draw our attention to the $C^*$-algebra generated by the monotone creators and annihilators.
Despite being by definition an object concretely realized in a given Hilbert space, we show that this $C^*$-algebra can
nevertheless  be recovered also as the universal $C^*$-algebra generated by a countable set of elements satisfying certain
relations. Accomplishing this task entails dispensing with those relations between annihilators and creators on the Fock space involving infinite sums whose convergence is only in the strong topology. Such a difficulty is overcome in Theorem \ref{universal}, where we show that a set of four rules displaying only finite sums are exactly
what is needed to give an abstract presentation of our $C^*$-algebra. We then move on to study the structure of the
$C^*$-algebra. In particular, we show that it is an approximately finite-dimensional $C^*$-algebra, namely  the norm 
closure of a union of an increasing sequence of finite-dimensional algebras. Moreover,  each of these algebras is a full matrix
algebra, and in Theorem \ref{Bratteli} we provide the relative  Bratteli diagram.
Once we have the diagram at hand, we draw some consenquences. For example, in Corollary \ref{simple} we show that the
(non-unital) monotone  $C^*$-algebra is simple, and in Proposition \ref{traceless} that no finite traces exist on it. 
Furthermore, we prove maximality of the (commutative) diagonal subalgebra obtained as the inductive limit of the diagonal subalgebras in each matrix block.
Finally,
we compute the $K$-theory of the $C^*$-algebra showing that its $K_0$ group coincides with the additive group of dyadic
numbers and that the scale is unbounded.

\section{Preliminaries}

For $k\geq 1$, set $I_k:=\{(i_1,i_2,\ldots,i_k) \mid i_1< i_2 < \cdots <i_k, i_j\in \mathbb{Z}\}$. The discrete monotone Fock space is the Hilbert space $\cf_{{\rm mon}}(\ch):=\bigoplus_{k=0}^{\infty} \ch_k$, where $\ch_0=\mathbb{C}$ and for every $k\geq 1$ $\ch_k:=\ell^2(I_k)$. Henceforth we will denote by $\Om$ the vector in $\cf_{{\rm mon}}(\ch)$
given by $(1, 0, 0, \ldots)$ and refer to it  as the Fock vacuum. As is commonly done, we call each
$\ch_k$ the $k$-particle space.
The canonical basis of the discrete monotone Fock space is obtained in the following way. If $(i_1,i_2,\ldots,i_k)\in I_k$ is an increasing sequence  of integers, we denote by $e_{(i_1,i_2,\ldots,i_k)}\in \ell^2(I_k)$ the square summable sequence that is always zero but at $(i_1, i_2, \ldots, i_k)$, where
it is $1$. The vector corresponding to the empty set
is just the vacuum $\Om$, and the corresponding vector state $\om_\Om$ is called the vacuum state. We will often write $e_i$ instead of $e_{(i)}$  to ease our notation.
For every $i\in \mathbb{Z}$, the monotone creation and annihilation operators are respectively given  by $a^\dag_i\Om=e_i$, $a_i\Om=0$ and
\begin{equation*}
a^\dagger_i e_{(i_1,i_2,\ldots,i_k)}:=\left\{
\begin{array}{ll}
e_{(i,i_1,i_2,\ldots,i_k)} & \text{if}\,\, i< i_1 \\
0 & \text{otherwise}, \\
\end{array}
\right.
\end{equation*}
\begin{equation*}
a_ie_{(i_1,i_2,\ldots,i_k)}:=\left\{
\begin{array}{ll}
e_{(i_2,\ldots,i_k)} & \text{if}\,\, k\geq 1\,\,\,\,\,\, \text{and}\,\,\,\,\,\, i=i_1\\
0 & \text{otherwise}. \\
\end{array}
\right.
\end{equation*}
Both $a^\dagger_i$ and $a_i$  have unital norm,
are mutually adjoint, and satisfy the following relations
\begin{equation}
\label{comrul}
\begin{array}{ll}
  a^\dagger_ia^\dagger_j=a_ja_i=0 & \text{if}\,\, i\geq j\,, \\
 
\end{array}
\end{equation}
In addition, for every $i$ the following identity holds
\begin{equation}
\label{comrul2}
a_ia^\dag_j=\delta_{i,j}\left(I-\sum_{k\leq i} a^\dag_ka_k\right)\,,
\end{equation}
where the convergence is in the strong operator topology, and $I$ is the identity
of $\cb(\cf_{{\rm mon}}(\ch))$, see {\it e.g.} \cite{CFL}.
We denote by $\gam$  the concrete unital
$C^*$-subalgebra of $\cb(\cf_{{\rm mon}}(\ch))$ with unit $I$ generated by the set of all annihilators $\{a_i\mid i\in\mathbb{Z}\}$, whereas  $\gam_0$ will denote the unital $*$-algebra of $\gam$ generated by the same set. Of course, $\gam_0$ is norm dense
in $\gam$. For completeness' sake, we recall that the algebra  generated by so-called position
operators $x_i:=a_i+a^\dag_i$  coincides with the whole $\gam$, see \cite[Proposition 5.13]{CFL}.
As the $x_i$'s have a symmetric two-point law with respect to $\om_\Om$, $\gam$ can be thought of as 
generated by Bernoulli monotone independent variables.
Recall that the pair $(\gam,\om_\Om)$ is a $C^*$-probability space, and the $\{a_i, a^\dag_i : i\in \bz\}$ are still monotone independent (algebraic) random variables. From now on $\gam$ will be referred to
as the monotone $C^*$-algebra.\\

In \cite{CFG}, among other things, a Hamel basis was exhibited for $\gam_0$. 
Since that basis will play a role in the present work as well, we next move on to describe what it looks like. 

\begin{defin}
A word $X$ in $\gam_0$ is said to have a $\lambda$-\textbf{form} if there are $m,n\in\left\{  0,1,2,\ldots
\right\}$ and $i_1<i_2<\cdots < i_m, j_1>j_2> \cdots > j_n$ such
that
$$
X=a_{i_1}^{\dagger}\cdots a_{i_m}^{\dagger} a_{j_1}\cdots a_{j_n}\,,
$$
with $X=I$, that is the empty word if $m=n=0$. Its length is $l(X)=m+n$.\\
A word $X$ is said to have a $\pi$-\textbf{form} if there are $m,n\in\left\{0,1,2,\ldots
\right\}$, $k\in\mathbb{Z}$, $i_1<i_2<\cdots < i_m, j_1>j_2> \cdots > j_n$ such that $i_m<k>j_1$ and
$$
X=a_{i_1}^{\dagger}\cdots a_{i_m}^{\dagger} a_{k}a_{k}^{\dagger} a_{j_1}\cdots a_{j_n}\,.
$$
The length is now $l(X)=m+n+2$.
\end{defin}
Let $\L$ be the index set such that $\{X_\l\}_{\l\in\L}$ represents all $\l$-forms, that is
\begin{equation*}
\label{xlambda}
X_\l=a_{i_1^{(\l)}}^{\dagger}\cdots a_{i_{m(\l)}^{(\l)}}^{\dagger} a_{j_1^{(\l)}}\cdots a_{j_{n(\l)}^{(\l)}}\,,
\end{equation*}
for $i_1^{(\l)}<i_2^{(\l)}<\cdots < i_{m(\l)}^{(\l)}, j_1^{(\l)}>j_2^{(\l)}> \cdots > j_{n(\l)}^{(\l)}$, $m(\l),n(\l)\geq 0$. Since all $\l$-forms of the type $a^\dag_ia_i$ are in one-to-one correspondence with $\bz$, then $\bz\subset\L$  in a natural way. After this identification, we put $\G:=\L\setminus\bz$. In \cite[Theorem 3.4]{CFG}  the families $\{X_\l\}_{\l\in\G}$ and $\{a_ia^\dag_i\}_{i\in \mathbb{Z}}$ were proved to form a Hamel basis of $\gam_0$.\\

\section{The monotone $C^*$-algebra}
\subsection{The monotone $C^*$-algebra as a universal $C^*$-algebra}
The main goal of this section is to come to
a characterization of the monotone $C^*$-algebra as a universal $C^*$-algebra generated by a countable set of elements
satisfying certain relations. To this aim, we start by noting that \eqref{comrul}--\eqref{comrul2} readily imply the following equalities involving the monotone creators and annihilators 
\begin{align}
\label{monrule1}&a_ia_j=0\,\, {\rm if}\,\, i\leq j\\
\label{monrule2}&a_ia_ja_j^\dag= \alpha_{i, j}a_i\\
\label{monrule3}&a^\dag_ia_i=a_{i-1}a^\dag_{i-1}- a_ia^\dag_i\\
\label{monrule4}&a_ia^\dag_i a_j=a_j-\alpha_{i, j}\sum_{j<k\leq i}a^\dag_ka_ka_j\,\, 
\end{align}
where $\alpha_{i,j}=1$ if $i>j$ and $\alpha_{i,j}=0$ otherwise.
We shall refer to the set of the four rules above as the monotone rules.

\begin{thm}\label{universal}
The universal unital $C^*$-algebra $\gb$ generated by a set $\{b_i: i\in\bz\}$ satisfying the monotone rules exists.
Moreover, the map $\Psi:\gb\rightarrow\mathfrak{M}$ such that $\Psi(b_i)=a_i$ for every $i\in\bz$ is
a $*$-isomorphism.
\end{thm}

\begin{proof}
Let $\gb_0$ be the universal $*$-algebra generated by the set $\{b_i: i\in\bz\}$ satysfying the monotone rules. By universality the map $b_i\mapsto a_i$, $i\in\bz$, extends to a representation
$\Psi_0: \gb_0\rightarrow \mathfrak{M}_0$.\\ 
We claim that $\Psi_0$ is injective. To prove this, 
we start by pointing out that Rules \eqref{monrule2}--\eqref{monrule4} imply the relations in Lemma 5.4 of \cite{CFL}, which in turn provide
the system of linear generators given by  $\lambda$-forms and $\pi$-forms, {\it cf.} \cite[Lemma 5.5]{CFL}.
Now  arguing as in the proofs of Lemma 3.2 and Lemma 3.3 in \cite{CFG}, one sees that
\eqref{monrule1}--\eqref{monrule3}--\eqref{monrule4} imply that
the family of all $\lambda$-forms, except for $b^\dag b_i$, and $b_ib^\dag_i$, $i\in\bz$,
is a Hamel basis of $\gb_0$ as a vector space.
For any integer $n\geq 1$, let $\gb_n\subset\gb_0$ be the (unital) $*$-subalgebra
generated by the finite set $\{b_i: |i|\leq n\}$.
$\gb_n$ is a finite-dimensional algebra because it is linearly generated by the elements in the Hamel basis
whose indices are picked from $\{i: |i|\leq n\}$. Denote by $\mathfrak{M}_n\subset\mathfrak{M}$ the (unital) $C^*$-subalgebra
generated by the finite set $\{a_i: |i|\leq n\}$.  Now $\mathfrak{M}_n$ is a finite-dimensional 
$C^*$-subalgebra, and $\Psi_0(\gb_n)=\mathfrak{M}_n$. Therefore, the restriction
$\Psi_0\lceil_{\gb_n}$ is injective because $\gb_n$ and $\mathfrak{M}_n$ have the same dimension.
Since $\gb_0=\cup_n\gb_n$, we see at once that $\Psi_0$ is injective.\\
Let now $\|\cdot\|_{\rm max}: \gb_0\rightarrow [0, \infty)$ be the
 the maximal $C^*$-seminorm on $\gb_0$. First note that
$\|\cdot\|_{\rm max}$ is actually a norm because $\gb_0$  has
a faithful representation.
Second, the maximal norm is well defined, namely it is bounded on each $b_i$. More precisely, there is exactly one $C^*$-norm on $\gb_0$.  Indeed, any $C^*$-norm is uniquely determined by its restriction on each subalgebra $\gb_n$.  But because all $\gb_n$'s are finite-dimensional, there can be only one such norm. \\  

Denote by $\gb$ the completion of $\gb_0$ w.r.t. the maximal norm above, and let  $\Psi: \gb\rightarrow\mathfrak{M}$ be the extension of $\Psi_0$. All is left to
do is show that $\Psi$ is still injective. This is a consequence of the fact that
$\Psi\lceil_{\gb_0}=\Psi_0$ is isometric.
\end{proof}
Henceforth we will denote by $\mathfrak{I}\subset\mathfrak{B}$ the non-unital $C^*$-algebra
generated by $\{b_i: i\in\bz\}$. Note that $\mathfrak{I}$ is in fact a closed two-sided ideal.

\begin{rem}
The restriction of the vacuum state to $\mathfrak{I}$ can also be characterized as the unique
shift-invariant state. More explicitly, the vacuum is the only state $\om$ on $\mathfrak{I}$ such that
$\om\circ\tau=\om$, where $\tau$ is the automorphism of $\mathfrak{B}$ uniquely determined
by $\tau(b_i)=b_{i+1}$, $i\in\bz$. See also \cite[Theorem 5.12]{CFL}.
\end{rem}

We end the section by showing that the Fock representation of
$\mathfrak{B}$ admits an intrinsic description
in terms of the relation \eqref{comrul2}, understood as an equality that holds at the level of
the von Neumann algebra generated by the representation.

\begin{lem}\label{limproj}
Any irreducible representation $\pi:\gb\rightarrow \cb(\ch_\pi)$ such that $\displaystyle{\lim_{i\rightarrow +\infty} \pi(b_ib_i^\dag)\neq 0}$ is unitarily equivalent
to the Fock representation.
\end{lem}

\begin{proof}
We first observe that $\pi(b_ib_i^\dag)$ is a monotone decreasing sequence of orthogonal projections
thanks to \eqref{monrule3}, so  $\displaystyle{\lim_{i\rightarrow +\infty} \pi(b_ib_i^\dag)}$ exists in the strong
operator topology and is an orthogonal projection, say $P$, which is nonzero by hypothesis.
Note that if $\xi$ lies in the range of $P$, then $\pi(b_i)\xi=0$ for every $i\in\bz$, because by \eqref{monrule1}--\eqref{monrule2}
$$\pi(b_i)\xi=\pi(b_i)\pi(b_ib_i^\dag)\xi=\pi(b_i^2b_i^\dag)\xi=0\,.$$
But then the vector state $\om_\xi$ associated with any such vector
coincides by construction with the Fock vacuum state. The conclusion follows by uniqueness (up to unitary equivalence)
of the GNS representation.
\end{proof}

\begin{prop}
Any irreducible representation $\pi:\gb\rightarrow \cb(\ch_\pi)$ such that $\displaystyle{\bigcap_{i\in\bz}{\rm Ker}\,\pi(b_i)\neq 0}$
and $\displaystyle{\pi(b_ib^\dag_i)=I-\sum_{k\leq i} \pi(b^\dag_kb_k)}$
is unitarily equivalent
to the Fock representation.
\end{prop}

\begin{proof}
A direct application of Lemma \ref{limproj}. Indeed, any (non-null) $\xi$ that sits in $\displaystyle{\bigcap_{i\in\bz}{\rm Ker}\,\pi(b_i)}$ also
satisfies $\pi(b_ib_i^\dag)\xi=\xi$ for every $i\in\bz$ by virtue of the equality $\displaystyle{\pi(b_ib^\dag_i)=I-\sum_{k\leq i} \pi(b^\dag_kb_k)}$.
\end{proof}

\subsection{The structure of the monotone $C^*$-algebra}

More information on $\mathfrak{B}$ is gained by looking more closely at the increasing sequence
of its subalgebras $\gb_n:=C^*(\{b_i: |i|\leq n\})$. These are not only finite-dimensional but also simple, as we next 
prove.

\begin{prop}\label{findim}
For every $n\in\bn$, the $C^*$-algebra $\gb_n:=C^*(\{b_i: |i|\leq n\})$ is a full matrix algebra. More precisely, one has
 $\gb_n\cong M_{2^{2n+1}}(\bc)$.
\end{prop}

\begin{proof}
It is enough to exhibit a faithful irreducible representation of $\gb_n$ on a finite-dimensional Hilbert space of dimension
$2^{2n+1}$.
To this end, let $V_n\subset \cf_{\rm mon}(\ch)$ be the finite-dimensional subspace generated by tensors 
$ e_{(i_1,i_2,\ldots,i_k)}$ whose indices $i_1, i_2, \ldots, i_k$ are all taken from the set $\{j: |j|\leq n\}$.
Note that $\Om$ corresponds to the empty set of indices and ${\rm dim}\, V_n=2^{2n+1}$. Now it is straightforward to see that $V_n$ is invariant for $\gb_n$.
The conclusion will then be achieved once we have shown that the restriction 
of the Fock representation (understood as a representation of $\gb_n$) to $V_n$, which we denote by $\r$, is both faithful and irreducible.\\
As for irreducibility, we start by observing that the vacuum vector $\Om$ is cyclic and
the projections $P_\Om$ of $V_n$ onto $\bc\Om$ belongs to $\r(\gb_n)$  since $P_\Om=\r(b_nb_n^\dag)$, as can be easily verified.
 The conclusion is now reached by the following standard argument.
Let us show that the commutant $\rho(\mathfrak{B}_n)'$ is trivial. Indeed, if
$T\in\rho(\mathfrak{B}_n)'$, one has $TP_\Om=P_\Om T$, hence $T\Om$ is in $\bc\Om$. Since $\Om$ is separating for 
$\rho(\mathfrak{B}_n)'$, one finds that $T$ must be a scalar.\\
In order to prove the faithfulness of $\r$, note that the  full monotone Fock space 
$\cf_{\rm mon}(\ch)$ decomposes into a direct sum of $\gb_n$-invariant spaces which are unitary equivalent
to either $\r$ or a degenerate representation.
But then $\r$ must be faithful because the Fock representation is so.
\end{proof}

\begin{cor}\label{simple}
$\mathfrak{I}$ is a simple $C^*$-algebra.
\end{cor}

\begin{proof}
It follows from Proposition \ref{findim} because $\mathfrak{I}$ is the inductive limit of simple $C^*$-algebras.
\end{proof}

\begin{cor}
$\mathfrak{I}$ is the only (proper) norm-closed two-sided ideal of the monotone $C^*$-algebra
$\gb$.
\end{cor}
\begin{proof}
Let $\mathfrak{K}$ be a norm-closed two-sided ideal of the monotone algebra.
Thanks to the above corollary, the intersection $\mathfrak{I}\cap\mathfrak{K}$ is either
$\mathfrak{I}$ or $\{0\}$. In the first case, $\mathfrak{K}$ contains $\mathfrak{I}$, thus
$\mathfrak{K}=\mathfrak{I}$ or $\mathfrak{K}=\gb$ since
$\mathfrak{I}$ has codimension one in $\gb$.\\
The second case can be dealt with in the Fock representation.
Denote by $E_n$ the unit of $\mathfrak{M}_n\cong M_{2^{2n+1}}(\bc)$.
The sequence $\{E_n:n\in\bn\}$ is contained
in $\mathfrak{I}$ and strongly converges to $I$. If now $T$ is any element of $\mathfrak{K}$, then
by assumption $TE_n=0$ for every $n$ since $TE_n\in \mathfrak{I}\cap \mathfrak{K}$, and thus $T=0$.
\end{proof}

As $\mathfrak{I}$ is an approximately finite-dimensional $C^*$-algebra, the full information on its
structure is encoded by the so-called Bratteli diagram, see {\it e.g.} Chapter 3 in \cite{D}.
Each $\mathfrak{B}_n$ is represented by a single bullet because $\mathfrak{B}_n$ is simple and thus is made up of only one
full-matrix block. Finally, each block is linked to the next by a number of arrows that equals the multiplicity
of the relative embedding, see again Chapter 3 in \cite{D}.

\begin{thm}\label{Bratteli}
The Bratteli diagram of $\mathfrak{I}$ is 

$$
\overset{2}{\bullet}\rightrightarrows
\overset{8}{\bullet}\rightrightarrows\cdots\rightrightarrows\overset{2^{2n+1}}{\bullet} \rightrightarrows\overset{2^{2n+3}}{\bullet}\rightrightarrows\cdots
$$

\end{thm}

\begin{proof}
We need only  say what the embedding of $\gam_n\cong M_{2^{2n+1}}(\bc)$ into $\gam_{n+1}\cong M_{2^{2n+3}}(\bc)$ is explicitly. This is best described in two steps.
To this aim, define $\mathfrak{C}_{n}=C^*(\{a_i: -n\leq i\leq n+1\})$. The same argument employed in the proof
of Proposition \ref{findim} shows that $\mathfrak{C}_{n}\cong M_{2^{2n+2}}(\bc)$ since
the restriction of the operators of  $\mathfrak{C}_{n}$ to 
the  subspace $W_n$ generated by tensors 
$ e_{(i_1,i_2,\ldots,i_k)}$ whose indices $i_1, i_2, \ldots, i_k$ are all taken from the set $\{j: -n\leq j\leq n+1\}$ is
a faithful irreducible representation.\\
By direct computation one finds that $W_n$ is no longer irreducible for $\gam_n$. In fact, up to unitary equivalence one has $W_n= V_n\oplus V_n$, as representations
of $\gam_n$. Phrased differently, the inclusion of $M_{2^{2n+1}}(\bc)\cong\gam_n\subset\mathfrak{C}_n\cong M_{2^{2n+2}}(\bc)$ is realized through an amplification of
multiplicity $2$.\\
As for the inclusion $\mathfrak{C}_n\subset\gam_{n+1}$, a direct computation shows that up to unitary equivalence one now has $V_{n+1}= W_n\oplus 0$,
as representations of $\mathfrak{C}_n$, where $0$ denotes the degenerate subspace (of dimension $2^{2n+2}$) of $\mathfrak{C}_n$ in $V_{n+1}$.

\end{proof}

\begin{prop}\label{traceless}
The monotone $C^*$-algebra $\gb$ is traceless.
\end{prop}

\begin{proof}
We shall argue by contradiction. Suppose $\tau$ is a trace defined on the whole monotone $C^*$-algebra.
Denote by $\tau_n$ the restriction of $\tau$ to the subalgebra $\gb_n\cong M_{2^{2n+1}}(\bc)$. 
Now the normalized trace is the only tracial state that a full matrix algebra can be endowed with. Therefore,
there exists a sequence $\{\lambda_n: n\in\bn\}$ of positive numbers such that
$\tau_n (\cdot)=\lambda_n {\rm Tr} (\cdot)$, where  ${\rm Tr}$ is the non-normalized trace of 
$M_{2^{2n+1}}(\bc)$. Since  by Proposition \ref{Bratteli} the embedding of $\gb_n$ into $\gb_{n+1}$ has multiplicity
$2$, the compatibility condition
$\tau_{n+1}\lceil_{\gb_n} =\tau_n$ entails $\lambda_n=2\lambda_{n+1}$ for every
$n$, which means $\lambda_n=\frac{1}{2^n}\lambda_0$. Let now $\{e_n:n\in\bn\}\subset\gb$ be the sequence where
$e_n$ is the unit of $\gb_n$.  One then finds at once $\tau(e_n)=\lambda_0\frac{2^{2n+1}}{2^n}=\lambda_0 2^{n+1}$.
But because $\|e_n\|=1$ for every $n$, $\tau$ cannot be bounded unless $\lambda_0=0$, in which
case $\tau$ is null.
\end{proof}

We denote by $\mathfrak{D}\subset\mathfrak{I}$ the commutative $C^*$-subalgebra obtained as the inductive limit of
the sequence $\mathfrak{D}_n\subset\mathfrak{B}_n\cong M_{2^{2n+1}}(\bc)$, where, for every
$n\in\bn$, $\mathfrak{D}_n$ is the subalgebra of diagonal matrices. As one would expect, as a consequence of
the UHF-structure of $\mathfrak{I}$, the commutative $C^*$-subalgebra $\mathfrak{D}$ turns out to be maximal 
abelian, namely it is not properly contained in any commutative subalgebra of $\mathfrak{I}$.

\begin{prop}\label{max}
$\mathfrak{D}$ is a maximal abelian subalgebra of $\mathfrak{I}$. 

\end{prop}

\begin{proof}
We shall think of $\mathfrak{I}$ as the concrete $C^*$-algebra acting on  the monotone  Fock space.
We start by observing that $\mathfrak{D}''$ coincides with the von Neumann algebra of all diagonal operators w.r.t. the canonical basis of $\cf_{\rm mon}(\ch)$. This can be easily seen by going back to the proof of Proposition \ref{findim} and noting that
the projections onto the finite-dimensional subspace $V_n$ converge strongly to $I$.\\
Let $E$ denote the canonical conditional expectation from $\cb(\cf_{\rm mon}(\ch))$ onto $\mathfrak{D}''$, 
{\it i.e. }$E[T]_{i, j}=T_{i, j}\delta_{i, j}$, for all $i, j\in\bz$ (where $\bz$ here indexes the elements of the basis of $\cf_{\rm mon}(\ch)$ ).
Now it is easy to see that $E[\mathfrak{I}]=\mathfrak{D}$.\\
We are ready to show that $\mathfrak{D}$ is maximal abelian in $\mathfrak{I}$, namely
$\mathfrak{D}'\cap\mathfrak{I}=\mathfrak{D}$. Obviously, it is only the inclusion 
$\mathfrak{D}'\cap\mathfrak{I}\subset\mathfrak{D}$ that needs to be dealt with. To this end, pick
$T\in \mathfrak{D}'\cap\mathfrak{I}=\mathfrak{D}''\cap\mathfrak{I}$. Since $T$ sits in $\mathfrak{I}$, there must exist
a sequence $\{T_k:k\in\bn\}$ such that $T_k\in \mathfrak{B}_{n_k}$ and $\|T_k-T\|\rightarrow 0$ for some sequence
$\{n_k:k\in\bz\}\in\bn$. But then by continuity $E[T_k]$ converges in norm to $E[T]$. Because $T$ also lies in 
$\mathfrak{D}''$, we have $E[T]=T$, hence $T$ belongs to $\mathfrak{D}$.
\end{proof}

Having the Bratteli diagram of $\mathfrak{I}$ also allows us to compute the $K$-theory of $\mathfrak{I}$.
For a quick account of $K$-theory for $C^*$-algebras the reader is referred to \emph{e.g.} Chapter $4$ of \cite{D}.
More precisely, in the next proposition we provide the ordered $K_0$ group as well as its scale. Denote by $\bz\left[\frac{1}{2}\right]$ the group of dyadic numbers.

\begin{prop}\label{ktheory}
$K_0(\mathfrak{I})= (\bz\left[\frac{1}{2}\right], \bz_+\left[\frac{1}{2}\right], \bz_+)$.
\end{prop}

\begin{proof}
As is well known, one has $K_0(M_k(\bc))= (\bz, \bz_+ , [0, k])$ for every $k\in\bn$.
In particular, $G_n:= K_0(\gb_n)= (\bz, \bz_+ , [0, 2^{2n+1}])$. By virtue of 
Proposition \ref{Bratteli}, the group homomorphisms
$\varphi_n: G_n\rightarrow G_{n+1}$ induced by the embedding of $\gb_n$ into $\gb_{n+1}$ are the multiplication by $2$. The result now easily follows from the continuity
of the $K_0$-functor w.r.t. inductive limits since $$\lim_{\rightarrow} G_n\cong \bz\left[\frac{1}{2}\right]\,.$$

\end{proof}

\section*{Acknowledgments}
\noindent
We acknowledge  the support of Italian INDAM-GNAMPA.

\end{document}